\documentclass[12pt]{article}

\usepackage[english]{babel}
\usepackage{amssymb}
\usepackage{amsmath}
\usepackage{amssymb}
\usepackage{amsthm}
\usepackage{latexsym}
\usepackage{amsfonts,amsmath,mathtools}
\usepackage{graphics}
\usepackage{float}
\usepackage{url}
\usepackage{subfig}
\usepackage{xcolor}
\usepackage[export]{adjustbox}
\usepackage{tikz-cd}
\usepackage{titlesec}
\usepackage{enumitem}

\usepackage[top=3.5cm,bottom=3.5cm,left=3.0cm,right=3.0cm,asymmetric]{geometry}

\usepackage{etoolbox}
\apptocmd{\thebibliography}{\setlength{\itemsep}{-3pt}}{}{}

\titleformat{\section}
{\centering\normalfont\scshape}{\thesection.}{0.5em}{}
\titleformat{\subsection}[runin]{\normalfont}{\thesubsection.}{0.5em}{\textbf}[]

\newtheorem{theorem}{Theorem}[section]
\newtheorem{lemma}[theorem]{Lemma}
\newtheorem{corollary}[theorem]{Corollary}

\newtheorem{proposition}[theorem]{Proposition}
\newtheorem{remark}[theorem]{Remark}

\newtheorem{example}[theorem]{Example}

\newtheorem{definition}[theorem]{Definition}

\def\depth{\textup{depth}}
\def\pd{\textup{proj}\phantom{.}\!\textup{dim}}
\def\indeg{\textup{indeg}}
\def\supp{\textup{supp}}

\newcommand\blfootnote[1]{%
	\begingroup
	\renewcommand\thefootnote{}\footnote{#1}%
	\addtocounter{footnote}{-1}%
	\endgroup
}

\begin{document}

\title{\bf\normalsize\MakeUppercase{Behaviour of the normalized depth function}}	
\author{\small ANTONINO FICARRA, J\"URGEN HERZOG, TAKAYUKI HIBI}	
\date{}
\maketitle

\begin{center}
	\begin{minipage}{0.85\linewidth}
		{\small\textsc{Abstract.} Let $I\subset S=K[x_1,\dots,x_n]$ be a squarefree monomial ideal, $K$ a field. The $k$th squarefree power $I^{[k]}$ of $I$ is the monomial ideal of $S$ generated by all squarefree monomials belonging to $I^k$. The biggest integer $k$ such that $I^{[k]}\ne(0)$ is called the monomial grade of $I$ and it is denoted by $\nu(I)$. Let $d_k$ be the minimum degree of the monomials belonging to $I^{[k]}$. Then, $\text{depth}(S/I^{[k]})\ge d_k-1$ for all $1\le k\le\nu(I)$. The normalized depth function of $I$ is defined as $g_{I}(k)=\text{depth}(S/I^{[k]})-(d_k-1)$, $1\le k\le\nu(I)$. It is expected that $g_I(k)$ is a non-increasing function for any $I$. In this article we study the behaviour of $g_{I}(k)$ under various operations on monomial ideals. Our main result characterizes all cochordal graphs $G$ such that for the edge ideal $I(G)$ of $G$ we have $g_{I(G)}(1)=1$. They are precisely all cochordal graphs $G$ whose complementary graph $G^c$ is connected and has a cut vertex. As a far-reaching application, for given integers $1\le s<m$ we construct a graph $G$ such that $\nu(I(G))=m$ and $g_{I(G)}(k)=0$ if and only if $k=s+1,\dots,m$. Finally, we show that any non-increasing function of non-negative integers is the normalized depth function of some squarefree monomial ideal.}
	\end{minipage}
\end{center}\normalsize
	\blfootnote{
		\hspace{-0,3cm} \emph{Keywords:} Normalized depth function, squarefree powers, matchings, edge ideals.
		
	}

\maketitle

\section{Introduction}

The study of the algebraic properties of the powers of a homogeneous ideal $I$ of a polynomial ring $S=K[x_1,\dots,x_n]$, $K$ a field, is a classical topic in Commutative Algebra. Many of the known invariants of $I$ behave asymptotically well, that is, stabilize or show a regular behaviour for sufficiently high powers of $I$.
In the last two decades the study of the depth function $f_I(k)=\depth(S/I^k)$ of a homogeneous ideal $I$ has attracted a lot of interest. A classical result of Brodmann \cite{B79} assures that $f_I(k)$ is constant for $k\gg0$. On the other hand, the initial behaviour of $f_I(k)$ remained quite elusive for a long time. It was conjectured in \cite{HH2005} that for any bounded convergent function $\varphi:\mathbb{Z}_{\ge0}\rightarrow\mathbb{Z}_{\ge0}$ there exists a suitable homogeneous ideal $I$ such that $\varphi=f_I$. Many years later, this conjecture was settled in the affirmative by H.T. H\`a, H. Nguyen, N. Trung and T. Trung in \cite[Theorem 4.1]{HNTT2021}.

Recently, the study of the depth function of squarefree powers of squarefree monomial ideals was initiated in \cite{EHHM2022b}, see also \cite{BHZN18,CFL,EH2021,EHHM2022a,SASF2022,SASF2023}. Let $I\subset S$ be a squarefree monomial ideal and $G(I)$ be its unique minimal set of monomial generators. The \textit{$k$th squarefree power} of $I$, denoted by $I^{[k]}$, is the monomial ideal generated by the squarefree monomials of $I^k$. Thus $u_1u_2\cdots u_k$, $u_i\in G(I)$, $i=1,\dots,k$, belongs to $G(I^{[k]})$ if and only if $u_1,u_2,\dots,u_k$ is a regular sequence. Let $\nu(I)$ be the \textit{monomial grade} of $I$, \emph{i.e.}, the biggest length of a monomial regular sequence contained in $I$. Then $I^{[k]}$ is non-zero if and only if $k\le\nu(I)$.

Our motivation for studying such powers also comes from graph theory. Let $G$ be a finite simple graph on vertex set $[n]=\{1,\dots,n\}$, \emph{i.e.}, $G$ has no loops or multiple or directed edges. Furthermore, all the graphs we consider in this article do not have isolated vertices. The \textit{edge ideal} $I(G)$, associated to $G$, is the ideal of $S=K[x_1,\dots,x_n]$ generated by all squarefree monomials $x_ix_j$, $i\ne j$, such that $\{i,j\}\in E(G)$. A \textit{matching} $M$ of $G$ is a set of edges of $G$ such that no two distinct edges of $M$ have common vertices. If $|M|=k$, then $M$ is called a \textit{$k$-matching}. We denote by $\nu(G)$ the \textit{matching number} of $G$, that is the biggest size of a matching of $G$. Then, if $u_p=x_{i_p}x_{j_p}\in I(G)$, $p=1,\dots,k$, we have that $u_1u_2\cdots u_k\in G(I(G)^{[k]})$ if and only if $M=\big\{\{i_p,j_p\}:p=1,\dots,k\big\}$ is a $k$-matching of $G$. In particular, $\nu(I(G))=\nu(G)$.

Again, let $I\subset S$ be a squarefree monomial ideal. We always let $S$ to be the smallest polynomial ring that contains $G(I)$. Our main object of study is the \textit{normalized depth function} of $I$. For $1\le k\le\nu(I)$, we denote by $d_k=\indeg(I^{[k]})$ the \textit{initial degree} of $I^{[k]}$, \emph{i.e.}, the minimum degree of a monomial generator of $I^{[k]}$. Then, for all $k\ge0$ such that $I^{[k]}\ne(0)$, we have $\depth(S/I^{[k]})\ge d_k-1$ \cite[Proposition 1.1(b)]{EHHM2022b}. The normalized depth function of $I$ is defined as
$$
g_I(k)=\depth(S/I^{[k]})-(d_k-1),\ \ \ \ \ k=1,\dots,\nu(I).
$$
In contrast to the behaviour of the depth function of ordinary powers, a quite different situation is expected. Indeed, it was predicted in \cite{EHHM2022b} that the following is true:\bigskip
\\
\textbf{Conjecture.} \textit{For any squarefree monomial ideal, $g_I(k)$ is a non-increasing function.}\bigskip

At present, this conjecture is widely open. In this article, we investigate the behaviour of the normalized depth function under some general operations and for a large class of edge ideals.\medskip

Let us discuss now the outlines of the article.

In Section 2, we discuss the behaviour of the normalized depth function with respect to two standard operations on monomial ideals: products and sums. In Theorem \ref{Thm:gproduct}, we show that the normalized depth function is additive if we take products of monomial ideals $I_1\subset S_1$, $I_2\subset S_2$ of polynomial rings $S_1,S_2$ in disjoint sets of variables. That is $g_{I_1I_2}(k)=g_{I_1}(k)+g_{I_2}(k)$. Hence, in Corollary \ref{Cor:gadditive} we deduce that $g_{I_1I_2}$ is non-increasing if both $g_{I_1}$, $g_{I_2}$ are non-increasing. Then, we analyze the relationship between $g_I$ and $g_{(I,x)}$ where $x$ is a variable not dividing any monomial generator of $I$. Under mild hypotheses, the precise relationship is obtained in Proposition \ref{Prop:(I,x)sqfrPowers}. Its proof depends on the concept of \textit{Betti splitting} \cite{FHT2009} and a criterion of Bolognini (Proposition \ref{Prop:Bol}). Next, if $g_I$ is non-increasing, then $g_{(I,x)}$ is non-increasing, too (Corollary \ref{Cor:gIgJincreasing}).\smallskip

Section 3 contains our main two results. We focus our attention on the class of cochordal graphs. Recall that a graph $G$ is called \textit{cochordal} if its \textit{complementary graph} $G^c$ is \textit{chordal}, that is, $G^c$ does not contain induced cycles of length greater than three. In 1990 \cite{F90}, Fr\"oberg proved that $I(G)$ has a linear resolution if and only if $G$ is cochordal. This result has been further refined by Herzog, Hibi and Zheng \cite{HHZ2004} by showing that $G$ is cochordal if and only if all ordinary powers $I(G)^k$ have linear quotients. It was noted in \cite[Corollary 3.2]{EHHM2022b} that all the squarefree powers $I(G)^{[k]}$ have linear quotients, $k=1,\dots,\nu(G)$, if $G$ is cochordal. Furthermore, in \cite[Corollary 2.2]{EHHM2022b} all graphs $G$ such that $g_{I(G)}(1)=0$ have been classified. In Theorem \ref{Thm:gIchordal} we classify all cochordal graphs $G$ such that $g_{I(G)}(1)=1$. They are precisely all cochordal graphs $G$ such that $G^c$ is connected with a cut vertex. Moreover, if $G$ is such a graph, the normalized depth function is $g_{I(G)}(1)=1$ and $g_{I(G)}(k)=0$ for $k=2,\dots,\nu(G)$. The proof of this theorem relies upon Hochster's formula and a criterion obtained in \cite{EHHM2022b} (Proposition \ref{Criterion:g=0}). An indispensable tool is the notion of \textit{special $k$-matching} (Definition \ref{Def:specialk-matc}), see also Example \ref{Ex:specialk-matc}. A far-reaching application of Theorem \ref{Thm:gIchordal} is given in  Theorem \ref{Thm:s<m}. Note that the conjecture on the non-increasingness of $g_I$ would also imply that if $g_{I}(k)=0$ then $g_{I}(k+1)=0$, too, for any $k<\nu(I)$. Hence, it is natural to consider the following problem, which was raised in \cite{EHHM2022b}.
\bigskip\\
\textbf{Problem.} \textit{For given integers $1\le s<m$, find a finite simple graph $G$ with $\nu(G)=m$ such that $g_{I(G)}(k)>0$ for $k=1,\dots,s$ and $g_{I(G)}(k)=0$ for $k=s+1,\dots,m$.}
\bigskip\\
In Theorem \ref{Thm:s<m} we solve the above problem. A variation of Proposition \ref{Prop:(I,x)sqfrPowers} (Lemma \ref{Lemma:adjoiningedge}) is required for its proof. In particular, for the graph $G$ we construct to solve the above problem, we have $g_{I(G)}(k)=s-(k-1)$ for $k=1,\dots,s$ and $g_{I(G)}(k)=0$ for $k=s+1,\dots,m$, $m=\nu(I(G))$. 

In Section 4, we show that any non-increasing sequence of non-negative integers can be the normalized depth function of some squarefree monomial ideal (Theorem \ref{Thm:AnyNon-Increasing}). On the other hand, it is an open question if any non-increasing function can be the normalized depth function of an edge ideal.\medskip

We gratefully acknowledge the use of \textit{Macaulay2} \cite{GDS} and CoCoA \cite{CoCoA} which have been invaluable tools to make our experiments.

\section{The behaviour of the normalized depth function with respect to some operations on monomial ideals}

In this section we analyze the behaviour of the normalized depth function with respect to some operations on monomial ideals.

Our first result shows that the normalized depth function is additive with respect to the product of monomial ideals in disjoint sets of variables.
\begin{theorem}\label{Thm:gproduct}
	Let $S_1=K[x_1,\dots,x_n]$ and $S_2=K[y_1,\dots,y_m]$ be polynomial rings in disjoint sets of variables and let $S=K[x_1,\dots,x_n,y_1,\dots,y_m]$. Let $I_1\subset S_1$, $I_2\subset S_2$ be squarefree monomial ideals. Then, $\nu(I_1I_2)=\min\{\nu(I_1),\nu(I_2)\}$ and for all $1\le k\le\nu(I_1I_2)$,
	$$
	g_{I_1I_2}(k)=g_{I_1}(k)+g_{I_2}(k).
	$$
\end{theorem}
\begin{proof}
	Obviously, $\nu(I_1I_2)=\min\{\nu(I_1),\nu(I_2)\}$. Let $1\le k\le\nu(I_1I_2)$. Note that $(I_1I_2)^{[k]}=I_1^{[k]}I_2^{[k]}$, and moreover $I_1^{[k]}\subset S_1$ and $I_2^{[k]}\subset S_2$. By \cite[Corollary 3.2]{HRR22},
	$$
	\pd(S/(I_1I_2)^{[k]})=\pd(S_1/I_1^{[k]})+\pd(S_2/I_2^{[k]})-1.
	$$
	Then
	$$
	n+m-\pd(S/(I_1I_2)^{[k]})=n-\pd(S_1/I_1^{[k]})+m-\pd(S_2/I_2^{[k]})+1.
	$$
	Let $d_k=\indeg(I_1^{[k]})$, $\delta_k=\indeg(I_2^{[k]})$. Then $\indeg((I_1I_2)^{[k]})=d_k+\delta_k$. Therefore, by the Auslander--Buchsbaum formula
	$$
	\depth(S/(I_1I_2)^{[k]})-(d_k+\delta_k-1)=\depth(S_1/I_1^{[k]})+\depth(S_2/I_2^{[k]})+1-(d_k+\delta_k-1),
	$$
	and hence $g_{I_1I_2}(k)=g_{I_1}(k)+g_{I_2}(k)$.
\end{proof}
\begin{corollary}\label{Cor:gadditive}
	Under the assumptions of the previous theorem, suppose $g_{I_1}$ and $g_{I_2}$ are non-increasing functions, then $g_{I_1I_2}$ is a non-increasing function too.
\end{corollary}

Let $I\subset S'=K[x_1,\dots,x_n]$ be a squarefree monomial ideal. Now we examine the relationship between $g_I$ and $g_J$, where $J=(I,x)\subset S=S'[x]=K[x_1,\dots,x_n,x]$.

For the proof of the next result we recall the concept of \textit{Betti splitting} \cite{FHT2009}.
Let $I$, $I_1$, $I_2$ be monomial ideals of $S$ such that $G(I)$ is the disjoint union of $G(I_1)$ and $G(I_2)$. We say that $I=I_1+I_2$ is a \textit{Betti splitting} if
\begin{equation}\label{eq:BettiNumberBS}
	\beta_{i,j}(I)=\beta_{i,j}(I_1)+\beta_{i,j}(I_2)+\beta_{i-1,j}(I_1\cap I_2) \ \ \ \textup{for all}\ i,j.
\end{equation}
In particular, by \cite[Corollary 2.2(a)]{FHT2009},
\begin{equation}\label{eq:pdBS}
	\pd(I)=\max\big\{\pd(I_1),\pd(I_2),\pd(I_1\cap I_2)+1 \big\}.
\end{equation}

The following criterion is due to Bolognini.
\begin{proposition}\label{Prop:Bol}
	\textup{\cite[Theorem 3.3]{DB}} Let $I$, $I_1$, $I_2$ be monomial ideals of $S$ such that $G(I)$ is the disjoint union of $G(I_1)$ and $G(I_2)$. Suppose that $I_1$ and $I_2$ are componentwise linear. Then $I=I_1+I_2$ is a Betti splitting.
\end{proposition}

\begin{proposition}\label{Prop:(I,x)sqfrPowers}
	Let $S'=K[x_1,\dots,x_n]$ and $S=S'[x]=K[x_1,\dots,x_n,x]$ be polynomial rings and let $I\subset S'$ be a squarefree monomial ideal all of whose squarefree powers are componentwise linear. Let $J=(I,x)$ and $d_k=\indeg(I^{[k]})$ for $1\le k\le\nu(I)$. Furthermore, set $g_I(0)=g_I(\nu(I)+1)=+\infty$ and $d_0=0$. Then $\nu(J)=\nu(I)+1$ and for all $1\le k\le\nu(J)$,
	\begin{equation}\label{eq:gJgImin}
		g_J(k)=\min\{g_I(k)+d_k-d_{k-1}-1,g_I(k-1)\}.
	\end{equation}
\end{proposition}
\begin{proof}
	Firstly we verify our formula in the cases $k=1$ and $k=\nu(J)$.\smallskip
	
	When $k=1$, then $J^{[1]}=J=I+(x)$ and $\depth(S/J)=\depth(S'/I)$. Since $\indeg(J)=\min\{d_1,1\}=1$, we get that
	$$
	g_J(1)=\depth(S/J)=\depth(S'/I)=\depth(S'/I)-(d_1-1)+(d_1-1)=g_I(1)+(d_1-1).
	$$
	This agrees with formula (\ref{eq:gJgImin}), since $d_0=0$ and $g_I(0)=+\infty$.\smallskip
	
	When $k=\nu(J)$, then $J^{[\nu(J)]}=xI^{[\nu(I)]}$. In this case, $\indeg(J^{[\nu(J)]})=\indeg(I^{\nu(I)})+1=d_{\nu(I)}+1$  and $\depth(S/J^{[\nu(J)]})=\depth(S'/I^{[\nu(I)]})+1$. Hence
	$$
	g_J(\nu(J))=\depth(S/J^{[\nu(J)]})-d_{\nu(I)}=\depth(S'/I^{[\nu(I)]})+1-d_{\nu(I)}=g_I(\nu(I)),
	$$
	and since $g_I(\nu(I)+1)=+\infty$, this agrees with (\ref{eq:gJgImin}).\smallskip
	
	Now let $1<k<\nu(J)$. Note that $J^{[k]}=I^{[k]}+xI^{[k-1]}$. By our hypothesis both ideals $I^{[k]}$ and $I^{[k-1]}$ are componentwise linear. Thus $xI^{[k-1]}$ is componentwise linear too, and by Proposition \ref{Prop:Bol}, $J^{[k]}=I^{[k]}+xI^{[k-1]}$ is a Betti splitting. Hence by equation (\ref{eq:pdBS})
	$$
	\pd(J^{[k]})=\max\big\{\pd(I^{[k]}),\pd(xI^{[k-1]}),\pd(I^{[k]}\cap xI^{[k-1]})+1 \big\}.
	$$
	Note that $I^{[k]}\subset I^{[k-1]}$ and since $x$ does not divide any of the minimal generators of $I$, we obtain that $I^{[k]}\cap xI^{[k-1]}=xI^{[k]}$. Since $\pd(S/xI^{[k]})=\pd(S'/I^{[k]})$, we have
	$$
	\pd(S/J^{[k]})=\max\big\{\pd(S'/I^{[k]})+1,\pd(S'/I^{[k-1]})\big\}.
	$$
	Applying the Auslander-Buchsbaum formula we get
	$$
	\depth(S/J^{[k]})=\min\big\{\depth(S'/I^{[k]}),\depth(S'/I^{[k-1]})+1\big\}.
	$$
	Note that
	$$
	\indeg(J^{[k]})=\min\{\indeg(I^{[k]}),\indeg(I^{[k-1]})+1\}=\min\{d_k,d_{k-1}+1\}=d_{k-1}+1
	$$
	because $d_k>d_{k-1}$. Hence
	\begin{align*}
		g_J(k)\ &=\ \depth(S/J^{[k]})-d_{k-1}\\
		&=\ \min\{\depth(S'/I^{[k]})-d_{k-1},\depth(S'/I^{[k-1]})-d_{k-1}+1\}\\
		&=\ \min\{\depth(S'/I^{[k]})-(d_k-1)+ d_k-d_{k-1}-1,\depth(S'/I^{[k-1]})-(d_{k-1}-1)\}\\
		&=\ \min\{g_I(k)+d_k-d_{k-1}-1,g_I(k-1)\},
	\end{align*}
	as desired.
\end{proof}
\begin{corollary}\label{Cor:gIgJincreasing}
	With the assumptions and notation of the previous proposition, it follows that $g_J$ is a non-increasing function if $g_I$ is non-increasing.
\end{corollary}
\begin{proof}
	By hypothesis $g_I(k+1)-g_I(k)\le 0$ for all $k=1,\dots,\nu(I)-1$. We shall prove that $g_J(k+1)-g_J(k)\le0$ for all $k=1,\dots,\nu(I)$. Recall that $\nu(J)=\nu(I)+1$.\medskip
	
	Let $k=1$. From Proposition \ref{Prop:(I,x)sqfrPowers} we have $g_J(2)=\min\{g_I(2)+d_2-d_1-1,g_I(1)\}$ and $g_J(1)=g_I(1)+(d_1-1)$. If $g_J(2)=g_I(1)$, then $g_J(2)-g_J(1)=-(d_1-1)\le0$. Otherwise, if $g_J(2)=g_I(2)+d_2-d_1-1$, then $g_I(1)\ge g_I(2)+d_2-d_1-1$ and
	\begin{align*}
		g_J(2)-g_J(1)\ &=\ g_I(2)+d_2-d_1-1-(g_I(1)+(d_1-1))\\
		&=\ g_I(2)+d_2-2d_1-g_I(1)\\
		&\le\ g_I(2)+d_2-2d_1-(g_I(2)+d_2-d_1-1)\\
		&=\ 1-d_1\le 0,
	\end{align*}
	since $d_1\ge1$.\medskip
	
	Let $k\in\{2,\dots,\nu(I)-1\}$. From Proposition \ref{Prop:(I,x)sqfrPowers} we have
	\begin{align*}
		g_J(k+1)\ &=\ \min\{g_I(k+1)+d_{k+1}-d_{k}-1,g_I(k)\},\\
		g_J(k)\ &=\ \min\{g_I(k)+d_k-d_{k-1}-1,g_I(k-1)\}.
	\end{align*}
	We distinguish the four possible cases.
	\smallskip\\
	\textsc{Case 1.} Assume $g_J(k+1)=g_I(k+1)+d_{k+1}-d_{k}-1$ and $g_J(k)=g_I(k)+d_k-d_{k-1}-1$. Then $g_J(k+1)\le g_I(k)$ and
	\begin{align*}
		g_J(k+1)-g_J(k)\ &=\ g_{J}(k+1)-g_I(k)-(d_k-d_{k-1}-1)\\
		&\le\ g_{I}(k)-g_I(k)-(d_k-d_{k-1}-1)\\
		&=\ -(d_k-d_{k-1}-1)\le0
	\end{align*}
	because $d_k\ge d_{k-1}+1$.
	\smallskip\\
	\textsc{Case 2.} Assume $g_J(k+1)=g_I(k+1)+d_{k+1}-d_{k}-1$ and $g_J(k)=g_I(k-1)$. Then $g_J(k+1)\le g_I(k)$ and $g_J(k+1)-g_J(k)\le g_I(k)-g_I(k-1)\le0$ by our hypothesis.
	\smallskip\\
	\textsc{Case 3.} Assume $g_J(k+1)=g_I(k)$ and $g_J(k)=g_I(k)+d_k-d_{k-1}-1$. Then we have $g_J(k+1)-g_J(k)=-(d_k-d_{k-1}-1)\le0$ as observed before.
	\smallskip\\
	\textsc{Case 4.} Assume $g_J(k+1)=g_I(k)$ and $g_J(k)=g_I(k-1)$. Then $g_J(k+1)-g_J(k)=g_I(k)-g_I(k-1)\le0$ by our hypothesis.\medskip
	
	Finally, let $k=\nu(I)$, then $g_J(k+1)=g_J(\nu(J))=g_I(\nu(I))=g_I(k)$. Whereas, $g_J(k)=\min\{g_I(k)+d_k-d_{k-1}-1,g_I(k-1)\}$. If $g_J(k)=g_I(k)+d_k-d_{k-1}-1$, then $g_J(k+1)-g_J(k)=-(d_k-d_{k-1}-1)\le0$, as noted before. Otherwise, if $g_J(k)=g_{I}(k-1)$, then $g_J(k+1)-g_J(k)=g_I(k)-g_{I}(k-1)\le0$, because by our hypothesis $g_I$ is a non-increasing function. The proof is complete.
\end{proof}

\begin{remark}
	\rm In Proposition \ref{Prop:(I,x)sqfrPowers} we assumed that all squarefree powers of $I$ are componentwise linear in order to guarantee that for all $2\le k\le\nu(I)$, $J^{[k]}=I^{[k]}+xI^{[k-1]}$ is a Betti splitting. However this hypothesis is not required because it was proved in the meanwhile that $J^{[k]}=I^{[k]}+xI^{[k-1]}$ is always a Betti splitting \cite[Lemma 1.4]{CFL}.
\end{remark}\smallskip

\section{The normalized depth function of cochordal graphs}

In this section we examine the normalized depth function of cochordal graphs.\medskip 

Let $G$ be a graph with vertex set $V(G)=[n]=\{1,\dots,n\}$ and edge set $E(G)$. We always assume that $G$ has no isolated vertices. Let $S=K[x_1,\ldots,x_n]$  be the polynomial ring in $n$ variables over a field $K$. The \textit{edge ideal} $I(G)$, associated to $G$, is the ideal of $S$ generated by the set of all monomials $x_ix_j$, $i\ne j$, for which $i$ is \textit{adjacent} to $j$, \emph{i.e.}, $\{i,j\}\in E(G)$.\smallskip

The graph $G$ is called \textit{connected} if for any $i$ and $j$ there is a \textit{path} between $i$ and $j$, that is, a sequence of vertices $i_0,i_1,\dots,i_p$ such that $i_0=i$, $i_p=j$ and $\{i_k,i_{k+1}\}\in E(G)$ for $k=0,\dots,p-1$. If $G$ is not connected it is called \textit{disconnected}. For any graph $G$ there exist unique subgraphs $C_1,\dots,C_t$, called the \textit{connected components} of $G$, such that $V(G)=\bigcup_{i=1}^{t}V(C_i)$, $E(G)=\bigcup_{i=1}^{t}E(C_i)$ and each $C_i$ is a connected graph. A vertex $i$ of a connected graph $G$ is called a \textit{cut vertex} of $G$, if $G-\{i\}$ is disconnected. For a graph $G$, one defines the \textit{complementary graph} $G^c$ by setting $V(G^c)=V(G)$ and $\{i,j\}$ is an edge of $G^c$ if and only if $\{i,j\}$ is not an edge of $G$. Finally, $G$ is called \textit{cochordal} if $G^c$ is \textit{chordal}, \emph{i.e.}, $G^c$ has no induced cycles of length bigger than three.\medskip

Our main result is the following,
\begin{theorem}\label{Thm:gIchordal}
	Let $G$ be a cochordal graph with no isolated vertices. Then, the following conditions are equivalent:
	\begin{enumerate}[label=\textup{(\roman*)}]
		\item $G^c$ is connected with a cut vertex;
		\item $g_{I(G)}(1)=1$;
		\item $g_{I(G)}(1)=1$ and $g_{I(G)}(k)=0$ for $k\ge2$.
	\end{enumerate}
\end{theorem}

For parts of the proof of this theorem we need to use Hochster's formula. Let $\Delta$ be a simplicial complex on the vertex set $[n]=\{1,2,\dots,n\}$, and let $I_\Delta$ be its Stanley-Reisner ideal. Recall that by Hochster's formula \cite[Theorem 8.1.1]{HHBook2011} we have
$$
\beta_i(S/I_{\Delta})=\sum_{W\subseteq[n]}\dim_K\widetilde{H}^{|W|-i-1}(\Delta_W;K),
$$
where $\widetilde{H}^j(\Delta_W;K)$ is the $j$th reduced simplicial cohomology module of the simplicial complex $\Delta_W=\{F\in\Delta:F\subseteq W\}$.\bigskip

Now, we explain the connection between simplicial complexes and squarefree powers of edge ideals. Recall that a \textit{matching} $M$ of a graph $G$ is a set of edges of $G$ such that no two distinct edges of $M$ have common vertices. If $|M|=k$, then $M$ is called a \textit{$k$-matching}. The vertex set $V(M)$ of $M$ is the set $\{i\in[n]:i\in e,\ \text{for some}\ e\in M\}$. The \textit{matching number} $\nu(G)$ of $G$ is the maximum of the sizes of the matchings of $G$. We have $\nu(G)=\nu(I(G))$. Thus $I(G)^{[k]}\ne(0)$ if and only if $k\le\nu(G)$. Let $G$ be a simple graph on $[n]$ and let $k=1,\dots,\nu(G)$. Then we define
$$
\Gamma_k(G)\ =\ \big\{F\subseteq V(G):\ V(M)\not\subseteq F\ \textup{for any}\ k\textup{-matching}\ M\ \textup{of}\ G\big\}.
$$
One has that $I(G)^{[k]}=I_{\Gamma_k(G)}$. In other words $\Gamma_k(G)$ is a simplicial complex on $V(G)$ whose Stanley-Reisner ideal is the $k$th squarefree power of $I(G)$. A case of particular interest occurs when $k=1$. Then $\Gamma_1(G)=\Delta(G^c)$ is the \textit{clique complex} of the complementary graph $G^c$ of $G$, and in particular $I(G)=I_{\Gamma_1(G)}=I_{\Delta(G^c)}$. A \textit{clique} $C$ of a graph $H$ is a subset of $V(H)$ such that for any $i,j\in C$, $i\ne j$, it follows that $\{i,j\}\in C$. The \textit{clique complex} of $H$ is the simplicial complex $\Delta(H)$ on vertex set $[n]$ whose faces are the cliques of $H$.

\begin{lemma}\label{Lem:Hi=0Chordal}
	Let $G$ be a chordal graph and let $\Delta(G)$ be its clique complex. Then $\widetilde{H}_i(\Delta(G);K)=\widetilde{H}^i(\Delta(G);K)=0$ for any $i\ne0$.
\end{lemma}
\begin{proof}
	By \cite[Corollary 9.2.2]{HHBook2011} we have that $\widetilde{H}_i(\Delta(G);K)=0$ for all $i\ne0$. Since $\dim_K\widetilde{H}_i(\Delta;K)=\dim_K\widetilde{H}^i(\Delta;K)$ for any simplicial complex $\Delta$ and any $i$, we also have that $\widetilde{H}^i(\Delta(G);K)=0$ for all $i\ne0$, as desired.
\end{proof}\medskip

The following result is well-known, see for example \cite[Problem 8.2]{HHBook2011}.
\begin{lemma}\label{Lem:H0connected}
	A simplicial complex $\Delta$ on $[n]$ is connected if and only if $\widetilde{H}^0(\Delta;K)=0$.
\end{lemma}

If $u\in S$ is a monomial, the set $\supp(u)=\{i:\ x_i\ \textup{divides}\ u\}$ is called the \textit{support} of $u$. Furthermore, a monomial ideal $I\subset S$ has \textit{linear quotients} if for some ordering $u_1,\dots,u_s$ of its minimal generating set $G(I)$, $(u_1,\dots,u_{j-1}):u_j$ is generated by variables, for all $j=2,\dots,s$.\smallskip

For our convenience we state the following results from \textup{\cite{F90}} and \textup{\cite[Corollary 3.2]{EHHM2022b}}.
\begin{proposition}\label{Prop:GFrobergCochordal}
	Let $G$ be a graph. Then
	\begin{enumerate}[label=\textup{(\alph*)}]
		\item  $I(G)$ has a linear resolution if and only if $G$ is cochordal.
		\item  If $G$ is cochordal, $I(G)^{[k]}$ has linear quotients for all $k=1,\dots,\nu(I(G))$.
	\end{enumerate}
\end{proposition}

To prove Theorem \ref{Thm:gIchordal} we need the concept of \textit{dominating $k$-matchings}. Recall that a $k$-matching $M$ of $G$ is a \textit{dominating $k$-matching} if $V(M)$ is a \textit{dominating set}, which means that any vertex $v\in V(G)-V(M)$ is adjacent to a vertex of $V(M)$.
\begin{proposition}\label{Criterion:g=0}
	\textup{\cite[Proposition 3.3]{EHHM2022b}} Let $G$ be a graph with no isolated vertices and $1\le k\le\nu(G)$. Suppose that $I(G)^{[k]}$ has linear quotients with respect to the ordering $u_1,\dots,u_s$ of its minimal monomial generators. Then the following statements are equivalent:
	\begin{enumerate}[label=\textup{(\roman*)}]
		\item $g_{I(G)}(k)=0$.
		\item There exist a dominating $k$-matching $M$ and some $i=2,\dots,s$ which satisfy the following conditions:
		\begin{enumerate}[label=\textup{(\alph*)}]
			\item $V(M)=\supp(u_i)$, and
			\item for any $t\in V(G)-V(M)$, there exists a $k$-matching $M'$ with $V(M')=\supp(u_m)$ for some $m=1,\dots,i-1$ such that $V(M')\subset V(M)\cup\{t\}$.
		\end{enumerate}
	\end{enumerate}
	In particular, if $G$ is a cochordal graph, the statements \textup{(i)} and \textup{(ii)} are equivalent.
\end{proposition}

Let us illustrate the previous criterion with an example.

\begin{example}\label{Ex:specialk-matc}
	\rm Consider the graph $G$ on vertex set $[6]=\{1,2,\dots,6\}$ depicted below.
	
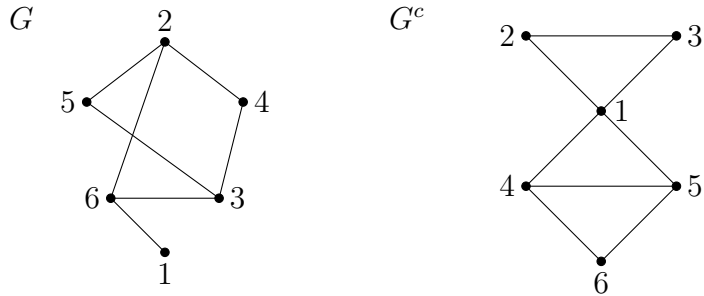
\begin{center}
		\begin{tikzpicture}[scale=0.8]
			\filldraw (0.1,0) circle (2pt) node[left]{6};
			\filldraw (1.9,0) circle (2pt) node[right]{3};
			\filldraw (-0.3,1.6) circle (2pt) node[left]{5};
			\filldraw (2.3,1.6) circle (2pt) node[right]{4};
			\filldraw (1,-0.9) circle (2pt) node[below]{1};
			\filldraw (1,2.6) circle (2pt) node[above]{2};
			\draw[-] (0.1,0) -- (1.9,0);
			\draw[-] (1.9,0) -- (2.3,1.6);
			\draw[-] (0.1,0) -- (1,-0.9);
			\draw[-] (1.9,0) -- (-0.3,1.6);
			\draw[-] (1,2.6) -- (0.1,0);
			\draw[-] (-0.3,1.6) -- (1,2.6);
			\draw[-] (2.3,1.6) -- (1,2.6);
			\filldraw (7,0.2) circle (2pt) node[left]{4};
			\filldraw (9.5,0.2) circle (2pt) node[right]{5};
			\filldraw (7,2.7) circle (2pt) node[left]{2};
			\filldraw (9.5,2.7) circle (2pt) node[right]{3};
			\filldraw (8.25,1.45) circle (2pt) node[right,xshift=0.7]{1};
			\filldraw (8.25,-1.05) circle (2pt) node[below]{6};
			\draw[-] (7,0.2) -- (9.5,0.2);
			\draw[-] (7,0.2) -- (9.5,2.7);
			\draw[-] (9.5,0.2) -- (7,2.7);
			\draw[-] (7,2.7) -- (9.5,2.7);
			\draw[-] (7,0.2) -- (8.25,-1.05);
			\draw[-] (9.5,0.2) -- (8.25,-1.05);
			\filldraw (-1,3) node[left]{$G$};
			\filldraw (5.5,3) node[left]{$G^c$};
		\end{tikzpicture}
	\end{center}\vspace*{-0.2cm}
	Note that $G$ is a cochordal graph, $G^c$ is connected with a cut vertex, namely 1. Since $G^c$ is chordal, by Proposition \ref{Prop:GFrobergCochordal}(b) all squarefree powers $I(G)^{[k]}$ have linear quotients. We have $\nu(G)=3$, $g_{I(G)}(1)=1$ and $g_{I(G)}(2)=g_{I(G)}(3)=0$. By using \textit{Macaulay2} \cite{GDS}, we obtained the following list. It provides a linear quotient order for $I(G)^{[1]}$, $I(G)^{[2]}$ and $I(G)^{[3]}$,
	\begin{align*}
		I(G)^{[1]}\ :\ \ \ &x_{2}x_{4},\:x_{3}x_{4},\:x_{2}x_{5},\:x_{3}x_{5},\:x_{2}x_{6},\:x_{3}x_{6},\:x_{1}x_{6};
		\\
		I(G)^{[2]}\ :\ \ \ &x_{2}x_{3}x_{4}x_{5},\:x_{2}x_{3}x_{4}x_{6},\:x_{1}x_{2}x_{4}x_{6},\:x_{1}x_{3}x_{4}x_{6},\\
		&x_{1}x_{2}x_{5}x_{6},\:x_{1}x_{3}x_{5}x_{6},\:x_{2}x_{3}x_{5}x_{6};
		\\
		I(G)^{[3]}\ :\ \ \ &x_{1}x_{2}x_{3}x_{4}x_{5}x_{6}.
	\end{align*}
	One can see each minimal generator $u\in I(G)^{[k]}$ as a $k$-matching. Fix $k=2$. Then we can order the generators of $I(G)^{[2]}$ as above. For instance $u_3=(x_1x_6)(x_2x_4)$ corresponds to the 2-matching $\big\{\{1,6\},\{2,4\}\big\}$. Consider now $u_6=(x_1x_6)(x_3x_5)$ and let $M=\big\{\{1,6\},\{3,5\}\big\}$. We claim that $M$ satisfies Proposition \ref{Criterion:g=0}(ii). Indeed $M$ is a dominating 2-matching and $V(G)-V(M)=\{2,4\}$. Let $M'=\big\{\{1,6\},\{2,5\}\big\}$, then $V(M')\subset V(M)\cup\{2\}$, $V(M')=\supp(u_5)$ and $5<6$. Likewise for $t=4$, we can consider $M''=\big\{\{1,6\},\{2,4\}\big\}$, then $V(M'')=\supp(u_3)\subset V(M)\cup\{2\}$ and $3<6$.\smallskip
	
	The vertex 1 in $G^c$ is a cut vertex and $G^c-\{1\}$ has two connected components: $C_1$ consisting only of the edge $\{2,3\}$ and $C_2$ consisting only of the triangle with vertices $4,5,6$. Note that in the matching $M=\big\{\{1,6\},\{3,5\}\big\}$, corresponding to $u_6\in I(G)^{[2]}$, the first edge arises by considering the cut vertex $1$ of $G^c$ and the second edge is an edge connecting the vertex $3\in V(C_1)$ to the vertex $5\in V(C_2)$. Furthermore $u_6$ is the biggest monomial in the given linear quotient order corresponding to this kind of matching, \emph{i.e.}, such that $\supp(u_6)=V(M)$ with $M$ such a matching.
\end{example}\medskip

We give a name to the kind of $k$-matchings we discovered in the previous example.\smallskip

Let $G$ be a graph such that $G^c$ is chordal, connected with a cut vertex. Let $i$ be a cut vertex of $G^c$. After a relabeling we can assume $i=1$. Then $G^c-\{1\}$ has at least two connected components. Let $C_1$ be one of these connected components and let $C_2$ be the union of all other connected components. Then $V(G^c-\{1\})=V(C_1)\cup V(C_2)$ and furthermore, for any $i\in V(C_1)$ and any $j\in V(C_2)$, $\{i,j\}\in E(G)$. The open neighbourhood $N_{G}(1)=\{j\in[n]-\{1\}:\{1,j\}\in E(G)\}$ is non-empty, otherwise $1$ would be an isolated vertex of $G$. Recall that we only consider graphs with no isolated vertices.
\begin{definition}\label{Def:specialk-matc}
	\rm Let $G$ be a simple finite graph such that $G^c$ is chordal, connected with cut vertex $1$. Let $k\ge2$. A $k$-matching $M=\{e_{1},e_{2},\dots,e_{k}\}$, $e_{p}=\{i_p,j_p\}\in E(G)$, $p=1,\dots,k$, is called \textit{special} if
	\begin{enumerate}[label=\textup{(\roman*)}]
		\item $e_1=\{1,j\}$ for some $j\in N_G(1)$, and
		\item $i_2\in V(C_1)$ and $j_2\in V(C_2)$.
	\end{enumerate}
\end{definition}\smallskip

\begin{lemma}\label{Lemma:specialk-matching}
	Let $G$ be a graph with no isolated vertices such that $G^c$ is chordal, connected with cut vertex $1$. Then for any $2\le k\le\nu(G)$, there exists a special $k$-matching of $G$.
\end{lemma}
\begin{proof}
	Let $M=\{e_{1},e_{2},\dots,e_{k}\}$, $e_{p}=\{i_p,j_p\}\in E(G)$, $p=1,\dots,k$ be an arbitrary $k$-matching. Firstly, we show that we can assume $e_1=\{1,j\}$ with $j\in N_G(1)$. If $1\in V(M)$ there is nothing to prove. Assume that $1\notin V(M)$. Let $N_{G}(1)=\{j\in[n]-\{1\}:\{1,j\}\in E(G)\}$ be the open neighbourhood of $1$ in $G$. As said before $N_{G}(1)$ is non-empty, otherwise $1$ would be an isolated vertex of $G$. If for some $j\in N_{G}(1)$, $j\in\supp(M)$, then $j=i_{q}$ for some $q$. We may assume $q=1$. Then $\big\{\{1,j\},\{i_2,j_2\},\dots,\{i_k,j_k\}\big\}$ is the desired $k$-matching of $G$. Otherwise, if $N_G(1)\cap V(M)=\emptyset$, then $(M-\{e_1\})\cup\big\{\{1,j\}\big\}$, with $j\in N_G(1)$, is the desired $k$-matching.
	
	Thus we may assume that $e_1=\{1,j\}$ with $j\in N_G(1)$. Now we prove that we can assume $i_{2}\in V(C_1)$ and $j_{2}\in V(C_2)$. We distinguish the two possible cases.
	\smallskip\\
	\textsc{Case 1.} Suppose that $V(M)-\{1\}\subseteq V(C_1)$. The case $V(M)-\{1\}\subseteq V(C_2)$ is analogous. Pick $i\in V(C_1)\setminus\{j\}$. Then we can consider the $k$-matching
	$$
	M'=(M-\{e_2\})\cup\big\{\{i,j_2\}\big\}=\big\{\{1,j\},\{i,j_2\},e_3,\dots,e_k\big\}.
	$$
	$M'$ is a special $k$-matching since $i\in V(C_1)$ and $j_2\in V(C_2)$.
	\smallskip\\
	\textsc{Case 2.} Suppose now that there exist $i\in V(C_1)$ and $j\in V(C_2)$ such that $\{i,j\}\subseteq V(M)$. If $i_q\in V(C_1)$ and $j_q\in V(C_2)$ for some $q$, then there is nothing to prove. Suppose that this is not true. Then since $k\ge2$ there exist integers $q_1$ and $q_2$ such that $i_{q_1},j_{q_1}\in V(C_1)$ and $i_{q_2},j_{q_2}\in V(C_2)$. Note that $q_1,q_2>1$ since $e_{1}=\{1,j\}$ and $1\notin C_1,C_2$. But then $\{i_{q_1},j_{q_2}\},\{j_{q_1},i_{q_2}\}$ are edges of $G$. Thus
	$$
	M'=(M-\{e_{q_1},e_{q_2}\})\cup\big\{\{i_{q_1},j_{q_2}\},\{j_{q_1},i_{q_2}\}\big\}
	$$
	is the desired special $k$-matching.\smallskip
	
	The cases above show that a special $k$-matching of $G$ exists.
\end{proof}
\begin{lemma}\label{Lemma:specialk-matching1}
	Let $G$ be a graph with no isolated vertices such that $G^c$ is chordal, connected with cut vertex $1$. Then a special $k$-matching is a dominating $k$-matching.
\end{lemma}
\begin{proof}
	Let $M=\{e_{1},e_{2},\dots,e_{k}\}$, $e_{p}=\{i_p,j_p\}\in E(G)$, $p=1,\dots,k$ be a special $k$-matching. Thus $e_1=\{1,j\}$ with $j\in N_G(1)$, $i_2\in V(C_1)$ and $j_2\in V(C_2)$. Let $t\in V(G)-V(M)$. Since $V(G)=V(C_1)\cup V(C_2)\cup\{1\}$ and $1\in V(M)$, either $t\in V(C_1)$ or $t\in V(C_2)$. If $t\in V(C_1)$, then $t$ is adjacent to $j_2\in V(C_2)$ and $\{i_2,j_2\}\in M$. Otherwise, if $t\in V(C_2)$, then $t$ is adjacent to $i_2\in V(C_1)$, as wanted.
\end{proof}

Now, we are in the position to prove Theorem \ref{Thm:gIchordal}.

\begin{proof}[Proof of Theorem \ref{Thm:gIchordal}]
	We are going to prove the implications (iii)$\Rightarrow$(ii), (ii)$\Rightarrow$(i) and (i)$\Rightarrow$(iii). The implication (iii)$\Rightarrow$(ii) is obvious.
	\medskip
	
	(ii)$\Rightarrow$(i): By the Auslander-Buchsbaum formula we know that (ii) is equivalent to $\pd(S/I(G))=n-2$. Since $I(G)=I_{\Delta(G^c)}$, by Hochster's formula
	\begin{align*}
		\beta_{n-2}(S/I(G))&=\ \ \ \sum_{W\subseteq[n]}\ \dim_K\widetilde{H}^{|W|-n+1}({\Delta(G^c)}_W;K)\\
		&=\sum_{\substack{W\subseteq[n]\\ n-1\le|W|\le n}}\!\!\!\dim_K\widetilde{H}^{|W|-n+1}({\Delta(G^c)}_W;K)
	\end{align*}
	must be non-zero. Here the last equation follows from the fact that $\widetilde{H}^j(\Delta;K)=0$ if $j<0$. Since $G^c$ is a chordal graph, by Lemma \ref{Lem:Hi=0Chordal} $\widetilde{H}^1(\Delta(G^c);K)=0$. Hence, the previous formula simplifies to
	\begin{align}
		\label{eq:beta(n-2)chordal}\beta_{n-2}(S/I(G))&=\sum_{j=1}^n\dim_K\widetilde{H}^{0}({\Delta(G^c)}_{[n]-\{j\}};K).
	\end{align}
	Since $\beta_{n-2}(S/I(G))$ is non-zero, there exists at least one integer $j\in[n]$ such that $\widetilde{H}^{0}({\Delta(G^c)}_{[n]-\{j\}};K)\ne0$, which means that $G^c$ has a cut vertex. Moreover, $G^c$ is connected by \cite[Corollary 2.2]{EHHM2022b}.
	\medskip
	
	(i)$\Rightarrow$(iii): Since $G^c$ is connected, by \cite[Corollary 2.2]{EHHM2022b} we have $g_{I(G)}(1)\ge1$. Since $G^c$ is chordal, under our assumptions equation (\ref{eq:beta(n-2)chordal}) holds. But $G^c$ has a cut vertex, which means that there exists a $j$ such that $\widetilde{H}^{0}({\Delta(G^c)}_{[n]-\{j\}};K)\ne0$ (Lemma \ref{Lem:H0connected}). Using formula (\ref{eq:beta(n-2)chordal}) this shows that $\beta_{n-2}(S/I(G))$ is non-zero and thus $g_{I(G)}(1)=1$.\smallskip
	
	It remains to prove that $g_{I(G)}(k)=0$ for all $k\ge2$. For this purpose, we use Proposition \ref{Criterion:g=0}. Since $G^c$ is chordal, by Proposition \ref{Prop:GFrobergCochordal}(b) all squarefree powers $I(G)^{[k]}$ have linear quotients. Let $2\le k\le\nu(G)$ and let $u_1,\dots,u_s$ be a linear quotient ordering for $I(G)^{[k]}$. By Lemma \ref{Lemma:specialk-matching} a special $k$-matching of $G$ exists. Let $i$ be the biggest integer such that $\supp(u_i)=V(M)$ with $M$ a special $k$-matching of $G$. Let $M=\{e_1,e_2,\dots,e_k\}$ be a special $k$-matching such that $\supp(u_i)=V(M)$. Assume the assumptions and notation before Definition \ref{Def:specialk-matc}. Then $e_1=\{1,j\}$ with $j\in N_G(1)$, $1$ is a cut vertex of $G^c$, $i_2\in V(C_1)$ and $j_2\in V(C_2)$. We claim that $M$ satisfies condition (ii) of Proposition \ref{Criterion:g=0}. Since $G$ is cochordal, this is equivalent to $g_{I(G)}(k)=0$ and will conclude our proof.
	
	By Lemma \ref{Lemma:specialk-matching1}, $M$ is a dominating $k$-matching. Let $t\in V(G)-V(M)$. Since $V(G^c-\{1\})=V(C_1)\cup V(C_2)$, then either $t\in V(C_1)$ or $t\in V(C_2)$. If $t\in V(C_1)$, then
	$$
	M'=(M-\{e_2\})\cup\big\{\{t,j_2\}\big\}
	$$
	is again a special $k$-matching, and $V(M')=\supp(u_m)$ for some $m$. By our assumption on $i$, we have $m<i$ and furthermore $V(M')\subset V(M)\cup\{t\}$.
	
	Similarly, if $t\in V(C_2)$ then
	$$
	M'=(M-\{e_2\})\cup\big\{\{i_2,t\}\big\}
	$$
	is special $k$-matching  with $V(M')=\supp(u_m)$, $m<i$ and $V(M')\subset V(M)\cup\{t\}$.
\end{proof}

Theorem \ref{Thm:gIchordal} has the following interesting consequence.
\begin{theorem}\label{Thm:s<m}
	Given positive integers $s<m$, there exists a graph $G$ with matching number $\nu(G)=m$ such that $g_{I(G)}(k)=0$ if and only if $k=s+1,\dots,m$.
\end{theorem}

For the proof of this theorem, we need the following lemma which is a variation of Proposition \ref{Prop:(I,x)sqfrPowers}.
\begin{lemma}\label{Lemma:adjoiningedge}
	Let $H$ be a cochordal graph on vertex set $[n]$ and let $G$ be the graph on vertex set $[n+2]$ whose edge set is $E(H)\cup\{\{n+1,n+2\}\}$. Then, $\nu(I(G))=\nu(I(H))+1$ and for all $k=1,\dots,\nu(I(G))$,
	$$
	g_{I(G)}(k)=\min\{g_{I(H)}(k)+1,g_{I(H)}(k-1)\},
	$$
	where we set $g_{I(H)}(0)=g_{I(H)}(\nu(I(G))+1)=+\infty$.
\end{lemma}
\begin{proof}
	The proof is very similar to that of Proposition \ref{Prop:(I,x)sqfrPowers}. We include a sketch.
	
	Our formula is easily verified for $k=1$ and $k=\nu(I(H))+1$. Let $2\le k\le\nu(I(H))$, then $I(G)^{[k]}=I(H)^{[k]}+x_{n+1}x_{n+2}I(H)^{[k-1]}$ is a Betti splitting, as both ideals $I(H)^{[k]}$ and $x_{n+1}x_{n+2}I(H)^{[k]}$ have linear resolutions, see Propositions \ref{Prop:GFrobergCochordal}(b) and \ref{Prop:Bol}. Since $I(H)^{[k]}\cap x_{n+1}x_{n+2}I(H)^{[k-1]}=x_{n+1}x_{n+2}I(H)^{[k]}$, by formula (\ref{eq:pdBS}) we have
	$$
	\pd(I(G)^{[k]})=\max\{\pd(I(H)^{[k]})+1,\pd(I(H)^{[k-1]})\}.
	$$
	Let $S=K[x_1,\dots,x_{n+2}]$ and $R=K[x_1,\dots,x_n]$. Then
	$$
	\depth(S/I(G)^{[k]})=\min\{\depth(R/I(H)^{[k]})+1,\depth(R/I(H)^{[k-1]})+2\}.
	$$
	Finally adding $-(2k-1)$ to both sides of the previous equation we obtain $g_{I(G)}(k)=\min\{g_{I(H)}(k)+1,g_{I(H)}(k-1)\}$, as desired.
\end{proof}
Using the exact same argument of the proof of Corollary \ref{Cor:gIgJincreasing} we get
\begin{corollary}
	Under the assumptions and notation of the previous lemma, if $g_{I(H)}$ is non-increasing, then $g_{I(G)}$ is non-increasing, too.
\end{corollary}

Now we are in the position to prove our second main result.
\begin{proof}[Proof of Theorem \ref{Thm:s<m}]
	For $s=1$ we can pick any graph $G$ with matching number $\nu(I(G))=m$ whose complementary graph satisfies condition (i) of Theorem \ref{Thm:gIchordal}. Then $g_{I(G)}(1)=1$ and $g_{I(G)}(k)=0$ for all $k=2,\dots,m$.
	
	Now, let $s>1$ and set $\ell=s-1$. Set $y_i=x_{n+2i-1}x_{n+2i}$, $i=1,\dots,\ell$. Let $G_0$ be any graph on vertex set $[n]$, $n$ big enough, whose complementary graph satisfies condition (iii) of Theorem \ref{Thm:gIchordal}, and with matching number $\nu(I(G_0))=m-\ell$. Let $R=K[x_1,\dots,x_n,y_1,\dots,y_\ell]$ and $J=(I(G_0),y_1,y_2,\dots,y_\ell)$. Then $\nu(J)=m$. We claim that $J^{[k]}$ has a linear resolution, $k=1,\dots,m$. For $\ell=1$, $J^{[k]}=I(G_0)^{[k]}+y_1I(G_0)^{[k-1]}$ is a Betti splitting, because $I(G_0)^{[k]}$, $y_1I(G_0)^{[k-1]}$ have linear resolutions. Note that $I(G_0)^{[k]}\cap y_1I(G_0)^{[k-1]}=y_1I(G_0)^{[k]}$ has again a linear resolution and it is equigenerated in degree $2k+1$. Since $J^{[k]}$ is equigenerated in degree $2k$, applying formula (\ref{eq:BettiNumberBS}) in our situation, we see that $J^{[k]}$ has again a linear resolution. For $\ell>1$, we set $L=(I(G_0),y_1,\dots,y_{\ell-1})$. Then $J=(L,y_\ell)$. By induction $L^{[k]}$ has a linear resolution, $k=1,\dots,\nu(L)$. Thus repeating the same argument as in the case $\ell=1$, it follows that $J^{[k]}$ has a linear resolution, for all $k=1,\dots,\nu(J)$.
	
	Let $S=K[x_1,\dots,x_n,x_{n+1},\dots,x_{n+2\ell}]$. Let $G$ be the graph on vertex set $[n+2\ell]$, $\ell=s-1$, whose edge set is
	$$
	E(G_0)\cup\big\{\{n+1,n+2\},\{n+3,n+4\},\dots,\{n+2s-3,n+2s-2\}\big\}.
	$$
	Note that $\nu(I(G))=m$. We claim that $g_{I(G)}(k)=s-(k-1)$ for $k=1,\dots,s$ and $g_{I(G)}(k)=0$ for $k=s+1,\dots,m$. This will conclude our proof.
	
	For $s=2$, our claim follows from Lemma \ref{Lemma:adjoiningedge}.
	Let $s>2$, $L=(I(G_0),y_1,\dots,y_{s-2})$ and $G'=G-\{n+2s-3,n+2s-2\}$. Note that $\pd(S/I(G)^{[k]})=\pd(R/J^{[k]})$. Since $J^{[k]}=L^{[k]}+y_\ell L^{[k-1]}$ is a Betti splitting, formula (\ref{eq:pdBS}) yields
	$$
	\pd(R/J^{[k]})=\max\big\{\pd(R/L^{[k]})+1,\pd(R/L^{[k-1]})\big\}.
	$$
	Thus
	$$
	\pd(S/I(G)^{[k]})=\max\big\{\pd(S/I(G')^{[k]})+1,\pd(S/I(G')^{[k-1]})\big\}.
	$$
	Arguing as in Proposition \ref{Prop:(I,x)sqfrPowers}, we have
	$$
	g_{I(G)}(k)=\min\big\{g_{I(G')}(k)+1,g_{I(G')}(k-1)\big\}.
	$$
	By induction on $s$, we may assume that $g_{I(G')}(k)=(s-1)-(k-1)$ for $k=1,\dots,s$ and $g_{I(G')}(k)=0$ for $k=s,\dots,m-1$. An easy calculation shows that $g_{I(G)}(k)=s-(k-1)$ for $k=1,\dots,s$ and $g_{I(G)}(k)=0$ for $k=s+1,\dots,m$, as desired.
\end{proof}

\section{Construction of non-increasing normalized depth functions}

We conclude this article with the following result which shows that any non-increasing function can be the normalized depth function of a suitable squarefree monomial ideal.
\begin{theorem}\label{Thm:AnyNon-Increasing}
	Let $a_1\ge a_2\ge\dots\ge a_m$ be a non-increasing sequence of non-negative integers. Then, there exists a squarefree monomial ideal $I\subset S=K[x_1,\dots,x_n]$, $n$ large enough, such that $\nu(I)=m$ and $g_{I}(k)=a_k$ for $k=1,\dots,m$.
\end{theorem}
For the proof of this result we need the following lemmata.
\begin{lemma}\label{Lem:AnyNon-Increasing1}
	Given  positive integers $s<m$, there exists a squarefree monomial ideal $I\subset S=K[x_1,\dots,x_n]$, $n$ large enough, such that $\nu(I)=m$, $g_I(k)=1$ for $k=1,\dots,s$ and $g_I(k)=0$ for $k=s+1,\dots,m$.
\end{lemma}
\begin{proof}
	Let $G$ be any graph on $[n]$, $n$ large enough, with matching number $\nu(G)=m-s$ such that $G^c$ is disconnected and chordal. Let $S=K[x_1,\dots,x_n,x_{n+1},\dots,x_{n+s}]$ and $I=(I(G),x_{n+1},\dots,x_{n+s})$. We claim that $\nu(I)=m$, $g_I(k)=1$ for $k=1,\dots,s$ and $g_I(k)=0$ for $k=s+1,\dots,m$. This will conclude our proof.
	
	By \cite[Corollary 2.6]{EHHM2022b}, $g_{I(G)}(k)=0$ for $k=1,\dots,m-s$. We prove our statement by induction on $s$. For $s=1$, since $\indeg(I(G))=2$ and $I(G)^{[k]}$ has a linear resolution for all $k=1,\dots,m-s$, by Proposition \ref{Prop:(I,x)sqfrPowers} we get that $\nu((I(G),x_{n+1}))=(m-s)+1$, $g_{(I(G),x_{n+1})}(1)=1$ and $g_{(I(G),x_{n+1})}(k)=0$ for $k=2,\dots,(m-s)+1$. Suppose our claim is true up to $s-1$ and set $L=(I(G),x_{n+1},\dots,x_{n+(s-1)})$. Then $L^{[k]}$ has a linear resolution, for $k=1,\dots,m-1$, and $I=(L,x_{n+s})$. Moreover, $\indeg(L)=1$, $\nu(L)=(m-s)+(s-1)=m-1$, $g_{L}(k)=1$ for $k=1,\dots,s-1$ and $g_{L}(k)=0$ for $k=s,\dots,m-1$. Applying again Proposition \ref{Prop:(I,x)sqfrPowers}, our statement follows.
\end{proof}
\begin{lemma}\label{Lem:AnyNon-Increasing2}
	Given a positive integer $m$, there exists a squarefree monomial ideal $I\subset S=K[x_1,\dots,x_n]$, $n$ large enough, such that $\nu(I)=m$ and $g_I(k)=1$ for $k=1,\dots,m$.
\end{lemma}
\begin{proof}
	By the previous lemma, we can find $L\subset K[x_1,\dots,x_n]$, $n$ large enough, such that $\nu(L)=m+1$, $g_{L}(k)=1$ for $k=1,\dots,m$, $g_{L}(m+1)=0$. Let $G$ be any graph on $[p]$ with matching $\nu(G)=m$ such that $G^c$ is disconnected. Let $I(G)=(y_iy_j:\{i,j\}\in E(G))$ be its edge ideal and $S=K[x_1,\dots,x_n,y_1,\dots,y_p]$. We claim that $I=L\cdot I(G)$ verifies $\nu(I)=m$ and $g_I(k)=1$ for $k=1,\dots,m$. This follows at once by Theorem \ref{Thm:gproduct} and \cite[Corollary 2.6]{EHHM2022b}.
\end{proof}\smallskip

Finally, we are in the position to prove our last result.

\begin{proof}[Proof of Theorem \ref{Thm:AnyNon-Increasing}]
	Any vector $(a_1,a_2,\dots,a_m)$, with $a_1\ge a_2\ge\dots\ge a_m\ge0$ integers, can be written uniquely as a sum of vectors of type $(1,1,\dots,1,0,0,\dots,0)$ and type $(1,1,\dots,1)$. Combining Theorem \ref{Thm:gproduct} with Lemmata \ref{Lem:AnyNon-Increasing1}, \ref{Lem:AnyNon-Increasing2}, the result follows.
\end{proof}

\subsection*{Acknowledgements}

We thank the referees for their helpful suggestions that improved the quality of the article. This article was written while the first author was visiting Department of Mathematics of University Duisburg-Essen, Germany. He would like to thank Professor Herzog for his support and hospitality. The third author was partially supported by JSPS KAKENHI 19H00637.

{\small\bibliographystyle{plain}

\begin{thebibliography}{1}
	
	\bibitem{CoCoA} J. Abbott, A. M. Bigatti, L. Robbiano, \textit{CoCoA: a system for doing Computations in Commutative Algebra}. Available at http://cocoa.dima.unige.it
	
	\bibitem{BHZN18} M. Bigdeli, J. Herzog, R. Zaare-Nahandi, \textit{On the index of powers of edge ideals}, Comm. Algebra, {\bf 46} (2018), 1080--1095.
	
	\bibitem{DB} D. Bolognini, \textit{Betti splitting via componentwise linear ideals}, J. of Algebra {\bf 455} (2016), 1-13.
	
	\bibitem{B79} M. Brodmann, \textit{The asymptotic nature of the analytic spread}, Math. Proc. Cambridge Philos. Soc., {\bf 86} (1979), 35–39.
	
	\bibitem{CFL} M. Crupi, A. Ficarra, E. Lax, \textit{Matchings, Squarefree Powers and Betti splittings}, 2023, available at \url{https://arxiv.org/abs/2304.00255}
	
	\bibitem{EH2021} N. Erey, T. Hibi, \textit{Squarefree powers of edge ideals of forests}, Electron. J. Combin., 28 (2) (2021), P2.32.
	
	\bibitem{EHHM2022a} N. Erey, J. Herzog, T. Hibi, S. Saeedi Madani, \textit{Matchings and squarefree powers of edge ideals}, J. Comb. Theory Series. A, {\bf 188} (2022).
	
	\bibitem{EHHM2022b} N. Erey, J. Herzog, T. Hibi, S. Saeedi Madani, \textit{The normalized depth function of squarefree powers}, Collect. Math. (2023). https://doi.org/10.1007/s13348-023-00392-x
	
	\bibitem{FHT2009} C. A. Francisco, H. T. Ha, A. Van Tuyl, \textit{Splittings of monomial ideals}, Proc. Amer. Math. Soc., 137 (10) (2009), 3271-3282.
	
	\bibitem{F90} R. Fr\"oberg, \textit{On Stanley-Reisner rings}, Topics in algebra, Banarch Center Publications, {\bf 26} (2) (1990), 57--70.
	
	\bibitem{GDS} D.~R.~Grayson, M.~E.~Stillman. {\em Macaulay2, a software system for research in algebraic geometry}. Available at \url{http://www.math.uiuc.edu/Macaulay2}.
	
	\bibitem{HNTT2021} H.T. H\`a, H. Nguyen, N. Trung, T. Trung, \textit{Depth functions of powers of homogeneous ideals}, Proc. AMS, {\bf 149} (2021), 1837--1844.
	
	\bibitem{HHBook2011} J.~Herzog, T.~Hibi. \emph{Monomial ideals}, Graduate texts in Mathematics {\bf 260}, Springer--Verlag, 2011.
	
	\bibitem{HH2005} J. Herzog, T. Hibi, \textit{The depth of powers of an ideal}, J. Algebra, {\bf 291} (2005), 534--550.
	
	\bibitem{HHZ2004} J. Herzog, T. Hibi and X. Zheng, \textit{Monomial ideals whose powers have a linear resolution}, Math. Scand. (2004), 23--32.
	
	\bibitem{HRR22} J. Herzog, M. Rahimbeigi, T. R\"omer, \textit{Classes of cut ideals and their Betti numbers}, to appear in S$\tilde{\text{a}}$o Paulo Journal of Mathematical Sciences.
	
	\bibitem{SASF2022} S. A. Seyed Fakhari, \textit{On the Castelnuovo-Mumford regularity of squarefree powers of edge ideals}, 2022, available at \url{arxiv.org/abs/2303.02791}
	
	\bibitem{SASF2023} S. A. Seyed Fakhari, \textit{On the Regularity of squarefree part of symbolic powers of edge ideals}, 2023, available at \url{arxiv.org/abs/2207.08559}
	
\end{thebibliography}

\textsc{Antonino Ficarra, Department of mathematics and computer sciences, physics and earth sciences, University of Messina, Viale Ferdinando Stagno d'Alcontres 31, 98166 Messina, Italy}

\textit{Email address:} \url{antficarra@unime.it}\smallskip\medskip

\textsc{J\"urgen Herzog, Fakult\"at f\"ur Mathematik, Universit\"at Duisburg-Essen, 45117 Essen, Germany}

\textit{Email address:} \url{juergen.herzog@uni-essen.de}\smallskip\medskip

\textsc{Takayuki Hibi, Department of Pure and Applied Mathematics, Graduate School of Information Science and Technology, Osaka University, Suita, Osaka 565--0871, Japan}

\textit{Email address:} \url{hibi@math.sci.osaka-u.ac.jp}
}

\end{document}